\documentclass[12pt,a4paper]{article}
\usepackage[english,frenchb]{babel}
\usepackage[latin1]{inputenc}
\usepackage{amsmath}
\usepackage{amsfonts}
\usepackage{amssymb}
\usepackage{amsthm}
\usepackage{hyperref}
\usepackage[all]{xy}
\usepackage[usenames,dvipsnames]{color}
\usepackage[dvips]{graphicx}

\newtheorem{maintheorem}{Theorem}

\newtheorem{mainpro}[maintheorem]{Proposition}

\newtheorem{teo}{Theorem}[section]

\newtheorem{pro}[teo]{Proposition}
\newtheorem{cor}[teo]{Corollary}
\newtheorem{lem}[teo]{Lemma}

\newtheorem{rem}[teo]{Remark}

\newtheorem{question}[teo]{Question}

\newcommand{\C}{\mathbb C}

\newcommand{\Hyp}{\mathbb H}
\newcommand{\Prob}{\mathbb P}
\newcommand{\tProb}{\widetilde{\mathbb P}}

\newcommand{\DD}{\mathcal D}
\newcommand{\FF}{\mathcal F}
\newcommand{\GG}{\mathcal G}

\newcommand{\CPun}{\mathbb{C}\mathbb{P}^1}
\newcommand{\CPdeux}{\mathbb{C}\mathbb{P}^2}
\newcommand{\PSLC}{PSL_2(\mathbb{C})}
\newcommand{\PSLR}{PSL_2(\mathbb{R})}
\newcommand{\SLC}{SL_2(\mathbb{C})}
\newcommand{\tS}{\widetilde{S}}

\title{On the dynamics of Riccati foliations with non parabolic monodromy representations.}
\author{Nicolas Hussenot Desenonges}

\begin{document}

\selectlanguage{english}
\maketitle

\begin{abstract}
In this paper, we study the dynamics of Riccati foliations over non-compact finite volume Riemann surfaces. More precisely, we are interested in two closely related questions: the asymptotic behaviour of the holonomy map $Hol_t(\omega)$ defined for every time $t$ over a generic Brownian path $\omega$ in the base; and the analytic continuation of holonomy germs of the foliation along Brownian paths in transversal complex curves. When the monodromy representation is parabolic (i.e. the monodromy around any puncture is a parabolic element in $\PSLC$), these two questions have been solved respectively in \cite{DD2} and \cite{Hus}. Here, we study the more general case where at least one puncture has hyperbolic monodromy. We characterize the lower-upper, upper-upper and upper-lower classes of the map $Hol_t(\omega)$ for almost every Brownian path $\omega$. And we prove that the main result of \cite{Hus} still holds in this case: when the monodromy group of the foliation is \og big enough \fg, the holonomy germs can be analytically continued along a generic Brownian path.
\end{abstract}



\let\thefootnote\relax

\section{Introduction.}

Let us consider the following differential equation in $\mathbb{C}^2$:
\begin{equation}\label{eq-dif}
\frac{dy}{dx}=\frac{P(x,y)}{Q(x,y)}
\end{equation}
where $P$ and $Q$ are polynomials in $\mathbb{C}[X,Y]$ without common factors.
The solutions of \eqref{eq-dif} define a singular holomorphic foliation of complex dimension $1$ in $\mathbb{C}^2$ which can be extended to a singular holomorphic foliation $\FF$ of $\CPdeux$. Very few is known about the dynamics of such a foliation. For example, we do not know if all the leaves accumulate on the singular set. We should mention, nevertheless, two important results of the theory: the minimality and Lebesgue-ergodicity of the leaves on an open set of the set $\FF_d$ of all foliations of degree $d$  \cite{LR}, the  existence and uniqueness of a harmonic current \cite{FS} for a generic foliation of $\FF_d$. To study the dynamics of a foliation, we can put a metric along the leaves and try to understand the asymptotic of a generic Brownian path in a leaf. It is the so called L. Garnett's theory of harmonic measures \cite{Gar}. This approach turned out to be very fruitful (at least in the non-singular case): see for example \cite{Ghy}, \cite{DK}.\newline

This paper was motivated by the two following questions:
\begin{enumerate}
\item What is the asymptotic behaviour of the holonomy map along a generic Brownian path in the leaves?
\item What can be said about the analytic continuation of holonomy germs of the foliation along a generic Brownian path in a transversal?
\end{enumerate}
In this paper, we restrict our study of these questions to the case of Riccati foliations that we think as examples of the general case of holomorphic foliations of $\CPdeux$. In this particular case, we answer partially to the first question and totally to the second one. Before stating the theorems, let us explain in more details the two questions in the general context of holomorphic foliations of $\CPdeux$.

\paragraph{Question 1:}Fix a Riemannian metric along the leaves varying countinuously in the transverse parameter. Then, we can define the Brownian motion along the leaves. If $L$ is a leaf and $\omega:[0,\infty[\rightarrow L$ is a Brownian path in $L$, take two discs $D_{\omega(0)}$ and $D_{\omega(t)}$ transverse to the foliation containing $\omega(0)$ and $\omega(t)$. If the discs are small enough, we can define a holonomy map $Hol_t(\omega):D_{\omega(0)}\rightarrow D_{\omega(t)}$ by \og sliding \fg along the leaves. A natural question is the following: what is the asymptotic of the map $Hol_t(\omega)$ for a generic path $\omega$? In the non-singular case, Deroin-Kleptsin \cite{DK} prove the following exponential contraction property: if $(M,\FF)$ is a compact manifold endowed with a transversely holomorphic foliation without singularities and if there is no transverse invariant measure, then there exists $\lambda<0$  such that for all $x$ in $M$ and almost every Brownian path $\omega$ starting et $x$, we have: 
\begin{equation}\label{DK}
\frac{1}{t}\log|Hol_t(\omega)'(x)|\underset{t\to\infty}\longrightarrow\lambda.
\end{equation}
It is known that a generic holomorphic foliation of $\CPdeux$ has no transverse invariant measure. Hence, it is natural to ask if a similar contraction property holds for a generic foliation of $\CPdeux$.

Very recently, V.A. Nguyen proved the existence of such a Lyapunov exponent $\lambda$ for a generic foliation of $\CPdeux$. Let us state this remarkable result. Let $\FF$ be a holomorphic foliation of $\CPdeux$ with hyperbolic singularities and without any invariant algebraic curve. According to Soares and Lins-Neto \cite{LNS}, these two conditions are generic. Under these assumptions, all the leaves are hyperbolic \cite{Glu}. Hence, there is a natural metric along the leaves: the Poincaré metric. According to a theorem of Candel and G\'omez-Mont, this family of metric along the leaves varies continuously in the transverse parameter \cite{CGM}. Moreover, according to \cite{FS}, the foliation supports a unique harmonic measure $\mu$ (in the sense of Garnett \cite{Gar}) associated with this metric. The result of Nguyen is the following \cite[Corollary 1.3]{Ngu}: there exists $\lambda\in\mathbb{R}$ such that \eqref{DK} holds for $\mu$ almost every base point $x$ and $\mathbb{P}_x$-almost every Brownian path starting at $x$.
Note that it is not known if $\lambda<0$.

\paragraph{Question 2:} The dynamics of a foliation can be encoded by the dynamics of its holonomy pseudo-group. The elements of this pseudo-group are local homeomorphisms between two small transversals. In our context, these holonomy maps are local biholomorphisms. Hence, it would be interesting to understand their nature with respect to analytic continuation. More precisely, take two non invariant algebraic lines $L_1$ and $L_2$ in $\CPdeux$ and let $h:(L_1,p_1)\to(L_2,p_2)$ be a holonomy germ of $\FF$. In \cite{L}, F. Loray conjectured that $h$ can be analytically continued along any path which avoids a countable set called the singularities (see also \cite{Il} for a conjecture in the same vein). In \cite{CDFG}, the authors show that, if the foliation has sufficiently rich contracting dynamics, Loray's conjecture does not hold. More precisely, they prove:
\begin{enumerate}
\item A generic foliation of $\CPdeux$ has a holonomy germ from an algebraic line to an algebraic curve whose singular set contains a Cantor set.
\item A Riccati foliation whose monodromy group is dense in $\PSLC$ has holonomy germs between algebraic lines with full singular set.
\end{enumerate}
Recall that a point $z$ is a \textit{singularitie} for analytic continuation of a germ $h:(L_1,p_1)\to(L_2,p_2)$ if there exists a continuous path $\alpha:[0,1]\rightarrow L_1$ with $\alpha(0)=p_1$, $\alpha(1)=z$ such that $h$ can be analytically continued along $\alpha([0,1-\varepsilon])$ for every $\varepsilon>0$ but $h$ cannot be analytically continued along $\alpha([0,1])$. According to Calsamiglia-Deroin-Frankel-Guillot, in general, there are germs $h:(L_1,p_1)\to(L_2,p_2)$ with a lot of singularities. Hence it is natural to ask if a generic Brownian path in $L_1$ leads to a singularitie. More precisely, for almost every Brownian path $\omega:[0,\infty[\to L_1$ with $\omega(0)=p_1$, does there exist a time $T(\omega)<\infty$ such that $h$ cannot be analytically continued along $\omega([0,T(\omega)])$?

\paragraph{Riccati foliations with non hyperbolic monodromy. }
These two questions have been solved in the case of Riccati foliations with non-hyperbolic monodromy. Let us explain the results in this case. Firstly, Recall that a Riccati equation is: 
\begin{equation}\label{eq-ric}
\frac{dy}{dx}=\frac{a(x)+b(x)y+c(x)y^2}{p(x)}
\end{equation}
where $a$, $b$, $c$ and $p$ are polynomials in $\mathbb{C}[X]$. Compactifying in $\CPun\times\CPun$ instead of $\CPdeux$, such foliations are characterized by the following geometric propertie: the foliation is transverse to all the vertical projective lines $\{x\}\times\CPun$ except for a finite set $\mathcal{S}$ (the zeroes of $p$) of values of $x$ for which $\{x\}\times\CPun$ is invariant for the foliation and contains the singularities. Hence a loop in $\CPun\setminus\mathcal{S}$ based in a point $a$ can be lifted via the projection on the first factor in the leaves of the foliation defining a biholomorphism of the fiber $\{a\}\times\CPun$. This holonomy map only depends on the homotopy class of the loop, hence we obtain a morphism $\rho:\pi_1(\CPun\setminus\mathcal{S})\to\PSLC$ called the monodromy representation of the foliation. 

More generally, we will call \textit {Riccati foliation} the data of a complex $2$-dimensional compact manifold $M$ which is a $\CPun$-fiber bundle over a compact Riemann surface $\Sigma$ endowed with a complex $1$-dimensional singular foliation $\FF$ which is transverse to all the fibers except a finite number of them which are invariant for the foliation. As in the case of Riccati equations in $\CPun\times\CPun$, a Riccati foliation comes with a monodromy representation $\rho:\pi_1(\Sigma\setminus\mathcal{S})\rightarrow\PSLC$ where $\mathcal{S}$ is the set of points in $\Sigma$ whose fibers are invariant. Later on, a Riccati foliation will be denoted $(M,\Pi,\FF,\Sigma,\rho,\mathcal{S})$ where $\Pi:M\rightarrow\Sigma$ is the natural projection of the fiber-bundle.

The monodromy representation is said to be \textit{non-hyperbolic} if for any $\alpha\in\pi_1(\Sigma\setminus\mathcal{S})$ making one loop around a puncture, the image $\rho(\alpha)$ is not a hyperbolic matrix in $\PSLC$ (i.e. $\rho(\alpha)$ is parabolic or elliptic).

The monodromy representation is said to be \textit{non-elementary} if there is no probability measure on $\CPun$ invariant by the action of the monodromy group $\rho(\pi_1(\Sigma\setminus\mathcal{S}))$.

In the context of Riccati foliations, the map $Hol_t$ is defined in the following way. Consider a smooth family of metrics with curvature $+1$ in the fibers compatible with their conformal structure. Assume that $\Sigma\setminus\mathcal{S}$ is hyperbolic. Endow it with the Poincaré metric. Fix a base point $a$ and consider a Brownian path $\omega$ starting at $a$. The leaves being transverse to the fibers over $\Sigma\setminus\mathcal{S}$, we can lift $\omega$ in the leaves. This defines a holonomy biholomorphism $Hol_t(\omega): \Pi^{-1}(\omega(0))\to\Pi^{-1}(\omega(t))$. Then we can define $||Hol_t(\omega)||$ by $\sup_{z\in\Pi^{-1}(\omega(0))} |Hol_t(\omega)'(z)|$ where the modulus of the derivative is with respect to the conformal metrics defined in the fibers.

In this case, the answer to Question $1$ is given in \cite[Proposition 2.5 and Proposition A]{DD2}: there exists a constant $\lambda> 0$ (the Lyapunov exponent), such that for almost every Brownian path $\omega$ starting at $a$: 
\begin{equation}
\label{Lyapunov}
\frac{1}{t}\log||Hol_t(\omega)||\underset{t\to\infty}\longrightarrow\lambda.
\end{equation}
The proof is essentially the following: denote by $\mathbb{P}(\cdot)=\int \mathbb{P}_x(\cdot)\cdot dvol(x)$ the probability measure on the set of infinite continuous paths in $\Sigma\setminus\mathcal{S}$ where $\mathbb{P}_x$ is the Wiener measure for the set of continuous paths starting at $x$ and $vol$ is the volume measure associated with the Poincaré metric in $\Sigma\setminus\mathcal{S}$. The norm $||\cdot||$ being submultiplicative, the family of functionals $\omega\mapsto\log||Hol_t(\omega)||$ defines a subadditive cocycle. If 
\begin{equation}
\label{integrability}
\int\log||A_1(\omega)||d\mathbb{P}(\omega)<+\infty,
\end{equation}
then by application of Kingman's subadditive ergodic theorem, we obtain a constant $\lambda$ (the Lyapunov exponent), such that \eqref{Lyapunov} is satisfied for almost every Brownian path $\omega$. It turns out that the integrability condition \eqref{integrability} is equivalent to the monodromy representation $\rho$ being non-hyperbolic. \\

Concerning the second question, we proved in \cite{Hus} that even when the set of singularities is big (Cantor or full), it is possible that the Brownian motion does not see this full set of singularities. More precisely, we proved:

\begin{teo}\label{theoRiccatiparabolic}
Let $(M,\Pi,\FF,\Sigma,\rho,\mathcal{S})$ be a Riccati foliation. Assume that the monodromy representation is non-elementary and non-hyperbolic. Let $F$ be a non invariant fiber and let $s_0$, $s_1$ be two holomorphic sections of the bundle. Denote by $S_0$ and $S_1$ the images of $\Sigma$ by $s_0$ and $s_1$. Endow $F$, $S_0$ and $S_1$ with complete metrics in their conformal class. Assume moreover that the developing map associated with $S_0$ is onto.
\begin{enumerate}
\item If $h:(F,p)\to (S_0,p_0)$ is a holonomy germ, then $h$ can be analytically continued along almost every Brownian path in $F$ starting at $p$.
\item If $h:(S_1,p_1)\to(S_0,p_0)$ is a holonomy germ, then $h$ can be analytically continued along almost every Brownian path in $S_1$ starting at $p_1$.
\end{enumerate}
\end{teo}

\begin{rem}
\begin{enumerate}
\item The theorem is stated in \cite{Hus} only in the parabolic case (i.e. when the image $\rho(\alpha)$ of any $\alpha\in\pi_1(\Sigma\setminus\mathcal{S})$ making one loop around a puncture is a parabolic matrix) but it is easy to see that it is still true if we allow $\rho(\alpha)$ to be elliptic.
\item One of the key points of the proof of the theorem is the existence and positivity of the Lyapunov exponent. 
\item The hypothesis \og the developing map associated with $S_0$ is onto\fg refers to the natural complex projective structure on $S_0\setminus s_0(\mathcal{S})$ induced by the foliation. The notion of complex projective structure and the link with Riccati foliations will be explained later. 
As it will be clear from the definition of the developing map $\DD_0$, if the action of the monodomy group $\rho(\pi_1(\Sigma\setminus\mathcal{S}))$ on $\CPun$ is minimal, then $\DD_0$ is onto (this is because the image of $\DD_0$ is an open set invariant by the action of $\rho(\pi_1(\Sigma\setminus\mathcal{S}))$). Hence, we can replace the hypothesis \og the developing map associated with $S_0$ is onto \fg in the theorem by \og the action of the monodomy group is minimal \fg in order to have hypothesis depending only on the monodromy of the foliation and not depending on the underlying complex projective structures.
\end{enumerate}
\end{rem}

\paragraph{The two main results.} 

Firstly, we give a (partial) answer to the first question for Riccati foliations whose base $\Sigma\setminus\mathcal{S}$ contains a loop $\alpha_0$ around a puncture such that $\rho(\alpha_0)$ is a hyperbolic matrix.

\begin{maintheorem}\label{theoHolonomie}
Let $(M,\Pi,\FF,\Sigma,\rho,\mathcal{S})$ be a Riccati foliation. Assume that the monodromy representation is non-elementary and that there exists a loop $\alpha_0$ around a puncture such that $\rho(\alpha_0)$ is a hyperbolic matrix. Let $h$ be a nondeacresing positive function. For all $a\in\Sigma\setminus\mathcal{S}$ and for almost every $\omega\in\Omega_a$, we have:
\begin{enumerate}
\item $$\liminf_{t\to\infty}\frac{\log||Hol_t(\omega)||}{t\cdot \log t}<+\infty.$$
\item If $\int_1^{\infty}\frac{dt}{h(t)}<+\infty$, then 
$$\limsup_{t\to\infty}\frac{\log||Hol_t(\omega)||}{h(t)}=0.$$
\item If $\int_1^{\infty}\frac{dt}{h(t)}=+\infty$, we have: 
$$\limsup_{t\to\infty}\frac{\log||Hol_t(\omega)||}{h(t)}=+\infty.$$
\end{enumerate}
\end{maintheorem}

This answer is only partial because this gives no information about the lower-lower class of $||Hol_t||$. We ask the following question:
\begin{question}
Do we have almost surely, $\liminf_{t\to\infty}\frac{\log||Hol_t||}{t\log t}>0$? It could be easier to answer the weaker following question: do we have
 $$\liminf_{t\to\infty}\frac{\log||Hol_t(\omega)||}{t}=+\infty?$$
\end{question}

Secondly, we give an answer to the second question for all Riccati foliations (with non-elementary monodromy representation):

\begin{maintheorem}\label{theoRiccati}
Theorem \ref{theoRiccatiparabolic} is still true if we delete the hypothesis: \og the monodromy representation is non-hyperbolic \fg.
\end{maintheorem}


\paragraph{On the length of the homotopic Brownian word. }
The proof of theorem \ref{theoHolonomie} is based on the study of the asymptotic of the length of the homotopic Brownian word in a hyperbolic non-compact finite volume Riemann surface. More precisely, let $S$ be such a Riemann surface. Fix a base point $a\in S$ and fix a set of generators of $\pi_1(S,a)$. Take a generic Brownian path (with respect to the Poincaré metric of $S$) starting at $a$. For every time $t$, append $\omega_{|[0,t]}$ with the smallest geodesic going back to $a$. This defines a loop based in $a$ which can be written in a unique way as a reduced word in the generators. Denote by $L_t(\omega)$ the length of this word. 
In the case where $S$ is the thrice punctured sphere, Gruet \cite{Gru} proved:
\begin{teo}[Gruet]
\label{Gruet}
\begin{enumerate}
\item Almost surely, 
$$\liminf\limits_{t\to\infty} \frac{L_t}{t\log t}=cst>0.$$
\item If $t\mapsto h(t)$ is a positive increasing function such that $\int_1^{+\infty}\frac{dt}{h(t)}<+\infty$, then almost surely:
$$\limsup\limits_{t\to\infty} \frac{L_t}{h(t)}=0.$$
\item If $t\mapsto h(t)$ is a positive increasing function such that $\int_1^{+\infty}\frac{dt}{h(t)}=+\infty$, then almost surely:
$$\limsup\limits_{t\to\infty} \frac{L_t}{h(t)}=+\infty.$$
\end{enumerate}
\end{teo}

The proof he gives of Items $2$ and $3$ (see \cite[Theorem 3.1]{Gru}) can be reproduced verbatim in our general case of a hyperbolic non-compact finite volume Riemann surface. We think that Item $1$ also holds in the general case. Nevertheless, the proof of Gruet is specific to the case of the thrice punctured sphere since he deduces his theorem from the analogous result for the geodesic flow and uses the Jacobi elliptic function as a universal cover of the thrice punctured sphere in order to make the link between the geodesic word length and the continued fraction digits problem. 
Using ideas already present in the paper of Gruet, we prove that the following weaker results hold: 

\begin{mainpro} 
\label{mainpro}
Let $S$ be a hyperbolic non-compact finite volume Riemann surface. Let $a$ be a base point in $S$ and consider the Wiener measure $\mathbb{P}_a$ associated with the Poincaré metric of $S$. Then, we have:
\begin{enumerate}
\item $\mathbb{P}_a$- almost surely, 
$$\liminf\limits_{t\to\infty} \frac{L_t}{t\log t}<\infty.$$
\item There exists $C>0$ such that:
$$\lim\limits_{t\to\infty} \Prob_a\left(\frac{L_t}{t\log t}\leq C\right)=1.$$
\end{enumerate}
\end{mainpro}

\begin{rem}
Note that even if we were able to prove that $\liminf\limits_{t\to\infty} \frac{L_t}{t\log t}>0$ in the general case of a hyperbolic non-compact finite volume Riemann surface, we would not know how to deduce from this that $\liminf_{t\to\infty}\frac{\log||Hol_t(\omega)||}{t\log t}>0$. 
\end{rem}

\paragraph{Complex projective structures. }
In the main theorem, there is an hypothesis concerning the developing map associated with $S_0$. This hypothesis relies on the notion of complex projective structures that we explain briefly now.
Let $S$ be a Riemann surface. A \emph{branched complex projective structure} on $S$ is a $(\PSLC,\CPun)$-structure on $S$. Equivalently, it is the data of a couple $(\DD,\rho)$ where $\DD$ is a non-constant holomorphic map $\DD:\tS\rightarrow \CPun$ from the universal cover of $S$ to $\CPun$ and $\rho$ is a morphism  $\rho:\pi_1(S)\rightarrow \PSLC$ satisfying the following equivariance relation:
$$ \forall \gamma \in \pi_1(S),\,  \DD\circ\gamma=\rho(\gamma)\circ\DD. $$
The map $\mathcal{D}$ is called \textit{developing map}, the morphism $\rho$ \textit{monodromy representation}, and the group $\rho(\pi_1(S))$ \textit{monodromy group}.

The main theorem \ref{theoRiccati} is a direct consequence of the following one concerning complex projective structures.

\begin{maintheorem}\label{theoprincipal}
Let $S=\Hyp/\Gamma$ be a hyperbolic Riemann surface of finite volume. Let $(\DD,\rho)$ be a complex projective structure on $S$.
Assume that $\mathcal{D}$ is onto and that the monodromy group $\rho(\Gamma)$ is non-elementary.

Then, for all $x_0\in\Hyp$, for almost every Brownian path $\omega$ starting at $x_0$ (for the Poincaré metric in $\Hyp$), $\mathcal{D}(\omega(t))$ does not have any limit when $t$ goes to $\infty$.

\end{maintheorem}

\paragraph{Organisation of the paper.} In section \ref{section2}, we prove some precise estimates making the link between the norm of an element of $\PSLC$ and the way this element acts geometrically on $\CPun$. These estimates will be used in section \ref{section4} and \ref{section5} where we prove respectively theorems \ref{theoHolonomie} and \ref{theoprincipal}. Section \ref{section3} is the probabilistic background: we explain the discretization procedure of Furstenberg-Lyons-Sullivan which associates to the Brownian motion on a finite volume Riemann surface, a random walk in its universal cover. In this section, we also prove proposition \ref{mainpro}. Finaly, in section \ref{section6}, we prove that theorem \ref{theoRiccati} is a direct consequence of theorem \ref{theoprincipal}.

\section{The action of $\PSLC$ on $\CPun$.}\label{section2}
In this section, we prove some easy lemmas about the action of an element of $\SLC$ on the Riemann sphere which will be usefull for the proof of theorems \ref{theoHolonomie} and \ref{theoprincipal}.\\

If $\gamma\in \SLC$, $||\gamma||=\underset{|| X ||=1}{sup}|| gX ||$, where $|| X ||$ is the euclidean norm of the vector $X\in\mathbb{C}^2$.
We introduce the chordal distance in $\CPun=\C\cup\{\infty\}$ defined by:
$$d(z_1,z_2)=\frac{|z_1-z_2|}{\sqrt{1+|z_1|^2}\sqrt{1+|z_2|^2}} \text{ if } z_1,z_2\in\C\subset\CPun,$$
$$d(z,\infty)=d\left(\frac{1}{z},0\right)=\frac{1}{\sqrt{1+|z|^2}}.$$
We denote by $D(z,\alpha)$ the open disc centered in $z$ with radius $\alpha$, by $D(z,\alpha)^c$ its complementary set and by $\overline{D(z,\alpha)}$ its closure.
Recall the Cartan decomposition: if $\gamma\in\SLC$, then there exists a triple $(k,k',a)$ where $k,k'$ belong to the special unitary group $SU_2(\mathbb{C})$, $a$ is the diagonal matrix
$a= \begin{pmatrix}
\lambda&0 \\
0&\lambda^{-1}
\end{pmatrix}$ with $|\lambda|\geq 1$ and $\gamma=kak'$. Moreover, the points $s=k'^{-1}(0)$ and $n=k(\infty)$ are uniquely determined by $\gamma$. It can be easily checked that $||\gamma||=|\lambda|$.

\begin{lem}\label{lemmetechnique1}
Let $\gamma\in\SLC$ with $||\gamma||>1$. Then, for every $\alpha\in]0,1[$:
\begin{equation}\label{equainf}
\gamma(D(s,\alpha)^c)=\overline{D(n,\alpha)}\Longleftrightarrow||\gamma||=\frac{\sqrt{1-\alpha^2}}{\alpha}
\end{equation}
and
\begin{equation}\label{equasup}
\gamma\left(\overline{D(s,\alpha)}\right)=D\left(n,\frac{1}{2}\right)^c\Longleftrightarrow\alpha=\sqrt{\frac{3}{3+||\gamma||^4}}.
\end{equation}
Furthermore, if $||\gamma||^4\geq\frac{3}{2}$, then:
\begin{equation}\label{equasup2}
\gamma\left(\overline{D\left(s,\frac{1}{||\gamma||^{2}}\right)}\right)\subset D\left(n,\frac{1}{2}\right)^c
\end{equation}
\end{lem}

\begin{proof}
Write the Cartan decomposition $\gamma=kak'$. The matrix $k$ and $k'$ belonging to $SU_2(\mathbb{C})$, they are isometries for the chordal distance $d$. We deduce that:
\begin{align*}
\gamma(D(s,\alpha)^c)=\overline{D(n,\alpha)}&\Longleftrightarrow a(D(0,\alpha)^c)=\overline{D(\infty,\alpha)}\\
&\Longleftrightarrow a(\partial D(0,\alpha))=\partial D(\infty,\alpha).
\end{align*}
But, we have 
\begin{equation}\label{equationn}
z\in\partial D(0,\alpha)\Longleftrightarrow d(z,0)^2=\alpha^2\Longleftrightarrow \frac{|z|^2}{1+|z|^2}=\alpha^2\Longleftrightarrow|z|^2=\frac{\alpha^2}{1-\alpha^2}
\end{equation}
and
$$a(z)\in\partial D(\infty,\alpha)\Longleftrightarrow d(a(z),\infty)^2=\alpha^2\Longleftrightarrow \frac{1}{1+|\lambda|^4|z|^2}=\alpha^2$$
Using the fact that $||\gamma||=|\lambda|$, a direct calculation gives \eqref{equainf}. A similar reasonning gives \eqref{equasup}. In order to prove \eqref{equasup2}, note that $||\gamma||^4\geq\frac{3}{2}\Longleftrightarrow\frac{1}{||\gamma||^2}\leq\sqrt{\frac{3}{3+||\gamma||^4}}$. So, if $||\gamma||^4\geq\frac{3}{2}$:
$$\gamma\left(\overline{D\left(s,\frac{1}{||\gamma||^2}\right)}\right)\subset\gamma\left(\overline{D(s,\sqrt{\frac{3}{3+||\gamma||^4}})}\right)=D\left(n,\frac{1}{2}\right)^c.$$
\end{proof}

\begin{lem}\label{lemmetechnique2}
Let $\gamma\in\SLC$ with $||\gamma||>1$. 
If there exist $y,z\in\CPun$ and $\alpha\in]0,\frac{1}{3}[$ such that $\gamma(D(y,\alpha)^c)\subset\overline{D(z,\alpha)}$ then:
$$||\gamma||\geq \frac{\sqrt{1-9\alpha^2}}{3\alpha}$$
\end{lem}

\begin{proof}
We still denote by $s$ (resp. $n$) the south pole (resp. north pole) of $\gamma$ obtained with Cartan decomposition.

Firstly, we claim that $D(y,\alpha)\cap D(s,\alpha)\neq\varnothing$. Indeed, assume on the contrary that $D(s,\alpha)\subset D(y,\alpha)^c$. Then, on the one hand, we have:
\begin{equation}
\label{equationn0}
\gamma(D(s,\alpha))\subset\gamma(D(y,\alpha)^c)\subset \overline{D(z,\alpha)}.
\end{equation}
And on the other hand, we have:
\begin{equation}\label{equationn1}
\gamma(D(s,\alpha))=kak'(D(s,\alpha))=ka(D(0,\alpha)).
\end{equation}
As $|\lambda|=||\gamma||>1$, there exists $\beta>1$ such that:
\begin{equation}\label{equationn2}
a(D(0,\alpha))\supset D(0,\beta\alpha).
\end{equation}
From \eqref{equationn1} and \eqref{equationn2}, we deduce:
$$\gamma(D(s,\alpha))\supset k(D(0,\beta\alpha))=D(k0,\beta\alpha).$$
Hence, using \eqref{equationn0}, we deduce that $D(k0,\beta\alpha)\subset \overline{D(z,\alpha)}$ which is clearly absurd since $\beta>1$.\newline

So we have $D(y,\alpha)\subset D(s,3\alpha)$. By the hypothesis of the lemma, this implies that $\gamma(D(s,3\alpha)^c)\subset \gamma(D(y,\alpha)^c)\subset \overline{D(z,\alpha)}$. As $\gamma(D(s,3\alpha)^c)=\overline{D(n,\epsilon)}$ for some $\epsilon>0$, we deduce that $\overline{D(n,\epsilon)}\subset \overline{D(z,\alpha)}$, hence $\epsilon\leq\alpha$. So $\gamma(D(s,3\alpha)^c)\subset \overline{D(n,\alpha)}\subset \overline{D(n,3\alpha)}$. So, according to lemma \ref{lemmetechnique1}, we have $||\gamma||\geq \frac{\sqrt{1-9\alpha^2}}{3\alpha}$.
\end{proof}

If $\gamma$ is a hyperbolic element of $\SLC$, we denote by $\lambda$ the eigenvalue with $|\lambda|>1$ and let $y_0$ (resp. $z_0$) be the repulsive (resp. attractive) fixed point of $\gamma$. We have:
\begin{lem}\label{lemmetechnique3}
 For every hyperbolic element $\gamma\in\SLC$, there exists $C_1>0$ such that for all $n\in\mathbb{N}$:
$$\gamma^n(D(y_0,C_1|\lambda|^{-n})^c)\subset \overline{D(z_0,C_1|\lambda|^{-n})}.$$
\end{lem}

\begin{proof}
Write $\gamma=P^{-1}aP$ with $P\in GL(2,\mathbb{C})$ and 
$a= \begin{pmatrix}
\lambda&0 \\
0&\lambda^{-1}
\end{pmatrix}$, $|\lambda|>1$.

$||a^n||=|\lambda|^n\geq \frac{\sqrt{1-|\lambda|^{-2n}}}{|\lambda|^{-n}}$. Hence according to lemma \ref{lemmetechnique1}:
$$a^n(D(0,|\lambda|^{-n})^c\subset \overline{D(\infty,|\lambda|^{-n})}.$$
Let $C_1=\sup_{x,y\in\CPun}\frac{d(P^{-1}(x),P^{-1}(y))}{d(x,y)}$. We have:
\begin{align*}
\gamma^n(D(y_0,C_1|\lambda|^{-n})^c)&=P^{-1}a^nP(D(y_0,C_1|\lambda|^{-n})^c)\\
&\subset P^{-1}a^n(D(0,|\lambda|^{-n})^c)\\
&\subset P^{-1}\overline{D(\infty,|\lambda|^{-n})}\\
&\subset \overline{D(z_0,C_1|\lambda|^{-n})}.
\end{align*}

\end{proof}

\begin{lem}\label{lemmetechnique4}
There is a universal constant $C_2$ such that for all $\gamma\in\SLC$ with $||\gamma||>1$, the restriction of $\gamma$ to $D(s,||\gamma||^{-1})^c$ is $C_2$-Lipschitz for the chordal distance.
\end{lem}

\begin{proof}
Let us write again the Cartan decomposition of $\gamma$: 
$\gamma=k_1ak_2$ with $k_1,k_2\in SU_2(\mathbb{C})$ and 
$a= \begin{pmatrix}
\lambda&0 \\
0&\lambda^{-1}
\end{pmatrix}$
The modulus of the derivative of $a$ in $z$ with respect to the spherical metric $\frac{|dz|}{1+|z|^2}$ writes:
$$|a'(z)|_{sph}=\frac{|a'(z)|_{eucl}\cdot (1+|z|^2)}{1+|a(z)|^2}=\frac{|\lambda|^2(1+|z|^2)}{1+|\lambda|^4|z|^2}.$$
If $z\in D(0,||\gamma||^{-1})^c$, then $|z|\geq||\gamma||^{-1}=|\lambda|^{-1}$ which is equivalent to:
$$|a'(z)|_{sph}\leq 1.$$
For $i=1,2$, we have $k_i\in SU_2(\mathbb{C})$, hence $|k_i'(z)|_{sph}=1$ for all $z\in\CPun$. Moreover, recall that $k_2(D(s,|\lambda|^{-1})^c)=D(0,|\lambda|^{-1})^c$ and $|\lambda|=||\gamma||$. We deduce that $|\gamma'(z)|_{sph}\leq 1$ for all $z\in D(s,|\lambda|^{-1})^c)$. Hence the restriction of $\gamma$ to $D(s,||\gamma||^{-1})^c$ is a contraction for the spherical distance associated with the spherical metric. The result follows from the fact that the chordal and spherical distances are equivalent.
\end{proof}

\section{The probabilistic background.}\label{section3}

\paragraph{Brownian motion. }
Let $M$ be a Riemannian manifold with complete metric. The Brownian motion on $M$ is defined as the Markov process with transition probability $p_t(x,y)dvol(y)$ where $dvol(y)$ is the volume element and the heat kernel satisfies:
\begin{itemize}
\item $\frac{\partial}{\partial t}p_t(x,\cdot)=\Delta p_t(x,\cdot)$
\item $\lim\limits_{t\to 0}p_t(x,.)=\delta_x$
\end{itemize}
This process is a diffusion with continuous sample paths. So it gives rise, for every $x\in M$, to a probability measure $\mathbb{P}_x$ (called the Wiener measure) on the set $\Omega_x$ of all continuous paths $\omega:[0,\infty[\to M$ such that $\omega(0)=x$. When $\nu$ is a measure on $M$, we denote by $P_{\nu}:=\displaystyle\int_M \Prob_x d\nu(x)$ the measure on the set $\Omega$ of all the continuous paths $\omega:[0,\infty[\to M$.

\paragraph{Discretization. }
In the most general context, this procedure associates to the Brownian motion in a Riemannian manifold $(M,g)$ a Markov chain in a discret $*$-recurrent set $E\subset M$ with time homogeneous transition probabilities \cite{LS}. Here, we explain the discretization in the case where $M=\Hyp$ is the universal covering space of a finite volume hyperbolic Riemann surface $S=\Hyp/\Gamma$ and $E=\Gamma\cdot x_0$ is the orbit of a base point $x_0$. We follow the presentation of \cite{KL}.\\

The fundamental group $\Gamma$ of $S$ acts on $\Hyp$ by isometry for the Poincar\'{e} metric of $\Hyp$. Fix a base point $x_0\in\Hyp$. For all $X \in\Gamma$, we define $F_X=X.\overline{D_{\Hyp}(x_0,\delta)}$ and $V_X=X.\overline{D_{\Hyp}(x_0,\delta')}$ with $\delta<\delta'$ and $\delta'$ small enough so that the $V_X$ are pairwise disjoints. Let $(\Omega_x,\mathbb{P}_x)$ be the set of Brownian paths starting at $x$ in $\Hyp$ with the Wiener measure associated with the Poincar\'{e} metric. $\underset{X \in \Gamma}{\bigcup}F_X$ is a recurrent set for the Brownian motion (because $S$ has finite volume). Let $X\in\Gamma$. For $x\in F_X$, consider $\varepsilon_x^{\partial V_X}$ the hitting measure of $\partial V_X$ for a Brownian motion starting at $x$. The Harnack constant $C_X$ of the couple $(F_X,V_X)$ is defined by:
\begin{equation}\label{C}
C_X=\sup \left\{\frac{d\varepsilon_x^{\partial V_X}}{d\varepsilon_y^{\partial V_X}}(z); x,y \in F_X, z\in\partial V_X\right\}
\end{equation}
where $\frac{d\varepsilon_x^{\partial V_X}}{d\varepsilon_y^{\partial V_X}}$ is the Radon-Nikodym derivative. Notice that, as elements of $\Gamma$ act isometrically on $\Hyp$, the Harnack constant of $(F_X,V_X)$ does not depend on $X \in \Gamma$ (i.e. there is a constant $C$ such that for all $X\in\Gamma$, $C_X=C$).\\

If $\omega \in \Omega_{x_0}$, we define recursively:
\[S_0(\omega)=\inf\left\{t\geq0 ; \omega(t)\in \partial V_{Id}\right\}\]
and, for $n\geq1$:
\[R_n(\omega)=\inf\left\{t\geq S_{n-1}(\omega) ; \omega(t)\in \cup F_X \right\}\]
\[S_n(\omega)=\inf\left\{t\geq R_{n}(\omega) ; \omega(t)\in \cup \partial V_X \right\}\]
We also define $X_n(\omega)$ by:
 \[X_0(\omega)=Id \text{ and } w(R_n(\omega)) \in F_{X_n(\omega)} \text{ for }n\geq 1\]
\begin{equation}\label{kappa}
\kappa_n(\omega)=\frac{1}{C}\left(\frac{d\epsilon_{X_n(\omega).0}^{\partial V_{X_n(\omega)}}}{d\epsilon_{\omega(R_n(\omega))}^{\partial V_{X_n(\omega)}}}(\omega(S_n(\omega)))\right)
\end{equation}
By definition of $C$ and $\kappa_n$, note that: 
\begin{equation}\label{kappa}
\frac{1}{C^2}\leq \kappa_n\leq 1.
\end{equation}

Now, define $(\Omega_{x_0}\times [0,1]^{\mathbb{N}},\mathbb{P}_{x_0}\otimes leb^{\otimes \mathbb{N}})=(\widetilde{\Omega},\widetilde{\mathbb{P}})$. Let
\[\begin{array}{ccccc}
N_k: & \widetilde{\Omega} & \longrightarrow & \mathbb{N}\\
 & (\omega,\alpha)=(\omega,(\alpha_n)_{n\in \mathbb{N}})=\widetilde{\omega} & \longmapsto &  N_k(\widetilde{\omega}) \\
\end{array}\]
 be the random variable defined recursively by:
\[N_0(\widetilde{\omega})=0\]
\[N_k(\omega,\alpha)=\inf\left\{n>N_{k-1}(\omega,\alpha); \alpha_n<\kappa_n(\omega)\right\}\]
The following theorem is stated in \cite{LS} in the cocompact case but it is observed in \cite[Proposition 4]{K} that it is also valid in the general set-up:

\begin{teo}\cite[Theorem 6]{LS}\label{theodiscretisation}
The distribution law of $X_{N_1}$ defines a probability measure $\mu$ on $\Gamma$ which satisfies for any Borel set $A$ in $\Hyp$:
\[\tilde{\mathbb{P}}(X_{N_1}=Y_1;\cdots;X_{N_k}=Y_k;\omega(S_{N_k})\in A)=\mu(Y_1) \mu(Y_1^{-1}Y_2)\cdots\mu(Y_{k-1}^{-1}Y_k)\varepsilon_{Y_k\cdot x_0}^{\partial V_{Y_k}}(A)\]
\end{teo}

\begin{cor}\cite{LS}
$(X_{N_k})_{k\in \mathbb{N}}$ is the realisation of a right random walk in $\Gamma$ with law $\mu$, in other words  $(\gamma_{N_k}:=X_{N_{k-1}}^{-1}X_{N_k})_{k\in \mathbb{N}^*}$ is a sequence of independent, identically distributed random variables with law $\mu$.
\end{cor}

We will need the following proposition:
\begin{pro}\label{propo}
\begin{enumerate}
\item There is a real $T>0$ such that $\tProb$ almost surely $\frac{S_{N_k}}{k}$ converges to $T$ when $k$ goes to infinity (see \cite[Corollaire 3.4]{KL}).
\item The measure $\mu$ is symmetric, i.e. for every $\gamma\in\Gamma$, $\mu(\gamma)=\mu(\gamma^{-1})$ (see \cite[Proposition 2.3]{BL}).
\end{enumerate}
\end{pro}

\paragraph{Random walks. }
The idea of the proof of theorem \ref{theoprincipal} (and of the similar theorem in \cite{Hus}) is to push forward (via the monodromy representation $\rho$) the measure $\mu$ obtained with the discretization procedure in order to get a right random walk in $\PSLC$. Let us recall briefly some basic facts about random walks in such matricial groups. The next two assertions were present in the original paper of Furstenberg \cite{Fur}. The reader could find a more recent presentation for example in \cite{BLa}. 
Let $G$ be a finitely generated subgroup of $\PSLC$. Let $\widetilde{\mu}$ be a probability measure on $G$. Choose randomly with law $\widetilde{\mu}$ and independently a sequence $(\gamma_n)_{n\geq 1}$ of elements of $G$ and consider the product $Y_n=\gamma_1\cdot\gamma_2\cdots\gamma_n$.

If the group generated by the support of $\widetilde{\mu}$ is non-elementary and $\widetilde{\mu}$ satisfies the integrability condition (the situation we have in \cite{Hus})
$$\displaystyle\int_{\gamma\in G}\log||\gamma||d\widetilde{\mu}(\gamma)<+\infty,$$
then there exists $\lambda>0$ such that almost surely, 
$$\lim\limits_{n\to\infty}\frac{1}{n}\log||Y_n||=\lambda.$$

When the integrability condition is not satisfied anymore (the situation we have in the present paper), we do not know at what speed $||Y_n||$ goes to $\infty$, but at least, we know that it goes:
\begin{pro}\label{proma}
If the group generated by the support of $\widetilde{\mu}$ is non-elementary, then almost surely, 
$$\lim\limits_{n\to\infty}||Y_n||=\infty.$$
\end{pro}

\paragraph{proof of proposition \ref{mainpro}.}
Let $S$ be a hyperbolic Riemann surface of finite volume with at least one puncture. Write $S=\Hyp/\Gamma$ where $\Gamma$ is a subgroup of $\PSLR$ and denote by $p:\Hyp\longrightarrow\CPun$ the universal covering map. Fix a base point $x_0\in\Hyp$ and let $F$ be the Dirichlet fundamental domain of $x_0$ (i.e. $F=\{x\in\Hyp\text{ s.t. } \forall \gamma\in\Gamma\text{; } d(x,x_0)\leq d(x,\gamma x_0)\}$). The set $\GG:=\{\gamma\in\Gamma\text{ s.t. } \gamma\cdot F\cap F\neq\varnothing\}$ is a (symmetric) set of generators of $\Gamma$. If $a,b\in\Hyp$, then there exists $\gamma_a,\gamma_b\in\Gamma$ such that $a\in\gamma_a\cdot F$ and $b\in\gamma_b \cdot F$. The element $\gamma_a^{-1}\gamma_b$ writes in an unique way as a reduced word in the elements of $\GG$. We denote by $L_{[a,b]}$ the length of this word. There is a little ambiguitie in this definition because $a$ (or $b$) could belong to the edges of the tesselation $(\gamma\cdot F)_{\gamma\in\Gamma}$ in which case there exist two elements $\gamma_a$ such that $a\in\gamma_a\cdot F$. Then $L_{[a,b]}$ can differ of at most two depending on the choices we make. But, this will not affect the next results since we will be interested in the asymptotic of a big word. To simplify the notations, if $\omega\in\Omega_{x_0}$ is a Brownian path (for the Poincar\'e metric) starting at $x_0$, we denote by $L_t(\omega)=L_{[\omega(0),\omega(t)]}$ and $L_{[t_1,t_2]}(\omega)=L_{[\omega(t_1),\omega(t_2)]}$. We want to prove that: 
\begin{enumerate}
\item $\Prob_{x_0}$-a.s.: $\liminf\limits_{t\to\infty} \frac{L_t}{t\log t}<\infty$ and
\item There exists $C>0$ such that $\lim\limits_{t\to\infty} \Prob_a\left(\frac{L_t}{t\log t}\leq C\right)=1.$
\end{enumerate}

In the sequel, it will be usefull to think the Brownian motion running in the surface $S$ instead of running in its universal cover $\Hyp$. For each puncture $p_{\alpha}$, let $d_{\alpha}$ and $D_{\alpha}$ be two closed horodiscs around $p_{\alpha}$ with $d_{\alpha}\subsetneq D_{\alpha}$, the $D_{\alpha}$ being choosen small enough so that $p(x_0)\notin \cup_{\alpha} D_{\alpha}$ and the $D_{\alpha}$ are pairwise disjoint. Assume moreover that all the $d_{\alpha}$ (resp. $D_{\alpha}$) are of the same size. For every $x\in S\setminus \cup_{\alpha} d_{\alpha}$ and every $\omega\in\Omega_x$, we define recursively two random sequences of times $U_n$ and $V_n$ by:
$$U_1(\omega)=\inf\left\{t\geq0 ; \omega(t)\in \cup_{\alpha} \partial d_{\alpha}\right\}$$
and, for $n\geq1$:
$$V_n(\omega)=\inf\left\{t\geq U_{n}(\omega) ; \omega(t)\in \cup_{\alpha} \partial D_{\alpha} \right\}$$
$$U_{n+1}(\omega)=\inf\left\{t\geq V_{n}(\omega) ; \omega(t)\in \cup_{\alpha} \partial d_{\alpha} \right\}$$
We have
\begin{equation}\label{ineq}
L_{V_n}\leq \sum_{k=1}^{n}{L_{[U_k,V_k]}}+L_{[0,U_1]}+\sum_{k=1}^{n-1}{L_{[V_k,U_{k+1}]}}
\end{equation}

\paragraph{Proof of Item 1.}
Firstly, the number of \og big\fg loops (i.e. the number of loops during the portions of trajectory between the $V_k$ and the $U_{k+1}$) before time $V_n$ is of the order of $n$. More precisely we have the following lemma:
\begin{lem}\label{lemmelongueur}
There is a constant $L>0$ such that $\Prob_{x_0}$-almost surely: 
$$\frac{L_{[0,U_1]}+\sum_{k=1}^{n-1}{L_{[V_k,U_{k+1}]}}}{n}\underset{n\to\infty}\longrightarrow L.$$
\end{lem}

\begin{proof}
We cannot apply directly the strong law of large numbers because the $(L_{[V_k,U_{k+1}]})_{k\geq 1}$ are not independent and not identically distributed with respect to $\mathbb{P}_{x_0}$. But we remark that $(\omega(V_k))_{k\geq 1}$ is a stationary Markov process whith values in the compact set $\cup_{\alpha}\partial D_{\alpha}$. Hence, there is a unique probability measure $\nu$ on $\cup_{\alpha}\partial D_{\alpha}$ which is stationary for the Markov process, i.e. which satisfies for every Borel set in $\cup_{\alpha}\partial D_{\alpha}$
$$\nu(A)=\displaystyle\int_{\cup_i\partial D_i}\Prob_{x}(\omega(V_1)\in A)d\nu(x).$$
Moreover, $\nu$ has a continuous positive density with respect to the measure $\varepsilon_{p(x_0)}$ defined by $\varepsilon_{p(x_0)}(A)=\Prob_{p(x_0)}(\omega(V_1)\in A)$. 
As $\nu$ is the unique stationary measure on $\cup_{\alpha}\partial D_{\alpha}$ for the Markov process,  $\mathbb{P}_{\nu}$ is an invariant ergodic measure for the time-shift $\sigma_{V_1}$. So if $L_{[0,U_1]}$ is $\mathbb{P}_{\nu}$-integrable, by Birkhoff ergodic theorem, we obtain the result $\mathbb{P}_{\nu}$-almost surely. The measure $\nu$ having a continuous positive density with respect to $\varepsilon_{p(x_0)}$, this implies the result $\mathbb{P}_{\varepsilon_{p(x_0)}}$-almost surely, which implies the result $\Prob_{p(x_0)}$-almost surely (and thus $\Prob_{x_0}$-almost surely) by the strong Markov property. It remains to prove that $L_{[0,U_1]}$ is $\mathbb{P}_{\nu}$-integrable :
\begin{align*}
\mathbb{E}_{\nu}[L_{[0,U_1]}]&\leq \underset{z\in\cup_{\alpha}\partial D_{\alpha}}\sup {E}_{z}[L_{[0,U_1]}]\\
                            &\leq cst\cdot\underset{\tau\in p^{-1}(\cup_{\alpha}\partial D_{\alpha})}\sup {E}_{\tau}[d_{\Hyp}(\omega(0),\omega(U_1))]
\end{align*}
The following inequality is classical (see for exemple \cite[lemma 2.11]{DD2}). There exist constants $K>0$ and $\alpha_1>0$ such that for all $\tau\in\Hyp$ and for all times $t\geq 0$:
\begin{equation}\label{sup}
\mathbb{P}_{\tau}\Big(\underset{0\leq s\leq t}\sup d_{\Hyp}(\omega(0),\omega(s))\geq K\cdot t\Big)\leq e^{-\alpha_1 t}
\end{equation}
On the other hand, there exists $0<p<1$ and a time $t_0$ such that for every $x\in S$, we have $\mathbb{P}_x\big(U_1\geq t_0\big)<p$. Then, by the strong Markov property, we deduce that for all $t$, we have:
\begin{align*}
\mathbb{P}_x\big(U_1\geq t_0+t\big)&\leq \mathbb{P}_x\big(U_1\geq t\big)\cdot \sup_{z\in S}\mathbb{P}_z\big(U_1\geq t_0\big)\\
&\leq p\cdot \mathbb{P}_x\big(U_1\geq t\big).
\end{align*}
This implies that there exist positive constants $C$, $\alpha_2$ such that for every $x\in\cup_{\alpha}\partial D_{\alpha}$:
\begin{equation}\label{U1}
\mathbb{P}_x\big(U_1\geq t\big)\leq Ce^{-\alpha_2 t}
\end{equation}
For every $t\geq 0$ and every $\tau\in p^{-1}(\cup_{\alpha}\partial D_{\alpha})$, we have:
$$\mathbb{P}_{\tau}\big(d_{\Hyp}(\omega(0),\omega(U_1))\geq K\cdot t\big)\leq \mathbb{P}_{\tau}\Big(\underset{0\leq s\leq t}\sup d_{\Hyp}(\omega(0),\omega(s))\geq K\cdot t\Big)+\mathbb{P}_{p(\tau)}\big(U_1\geq t\big)$$
Thus, according to inequalities \eqref{sup} and \eqref{U1}, this quantitie decreases exponentially fast with $t$.
This implies that for all $\tau\in p^{-1}(\cup_{\alpha}\partial D_{\alpha})$, ${E}_{\tau}[d_{\Hyp}(\omega(0),\omega(U_1))]$ is finite and thus $\mathbb{E}_{\nu}[L_{[0,U_1]}]$ is also finite.
\end{proof}

Secondly, for every $\alpha$, there is a conformal bijective map from $D_{\alpha}$ to $D_{eucl}(0,1)\setminus\{0\}\subset \mathbb{C}$ identifying $d_{\alpha}$ with some disc $D_{eucl}(0,r)\setminus\{0\}$ ($0<r<1$). For an Euclidean Brownian motion starting at a point of $\partial D_{eucl}(0,r)$, let $\theta_t$ be a local continuous determination of the winding around $0$. If $V$ denote the hitting time of $\partial D_{eucl}(0,1)$, then it is classical that the law of $\theta_V$ is that of a Cauchy law with paramater $a=\log\frac{1}{r}$ (i.e. the distribution law of $\theta_V$ is $\frac{1}{\pi}\cdot\frac{a}{a^2+x^2}$). Moreover, it is well known that a sequence $X_n$ of i.i.d. random variables with Cauchy laws with parameter $a$ satisfies the following (see \cite{FP}):
$$\liminf\limits_{n\to\infty} \frac{\sum_{k=1}^{n}|X_k|}{n\log n}\underset{n\to\infty}\longrightarrow a\cdot\frac{2}{\pi}$$
From this and the conformal invariance of the Brownian motion, we deduce that:  
$$\liminf\limits_{n\to\infty} \frac{\sum_{k=1}^{n}{L_{[U_k,V_k]}}}{n\log n}\underset{n\to\infty}\longrightarrow cst$$
From this, lemma \ref{lemmelongueur} and inequalitie \eqref{ineq} we deduce that:
$$\liminf\limits_{n\to\infty} \frac{L_{V_n}}{n\log n}<+\infty$$
Using exactly the same arguments as in the proof of proposition \ref{propo} given by Karlsson-Ledrappier, we can prove that there exists a constant $S>0$ such that $\Prob_{x_0}$-almost surely, $\frac{V_n}{n}$ converges to $S$ when $n$ goes to infinity, this implies that:
$$\liminf\limits_{n\to\infty} \frac{L_{V_n}}{V_n\log V_n}<+\infty$$
and proposition \ref{mainpro} Item $1$ is proved.\\

\paragraph{proof of Item 2.} Fix $t>0$. Assume firstly that $t\in[V_n,U_{n+1}]$ for some $n$. Then:
$$L_t\leq \sum_{k=1}^{n}{L_{[U_k,V_k]}}+L_{[0,U_1]}+\sum_{k=1}^{n-1}{L_{[V_k,U_{k+1}]}}+L_{[V_n,t]}.$$
We have seen that $L_{[U_k,V_k]}=|X_k|$ where $(X_k)_{k\geq 0}$ is a sequence of i.i.d. Cauchy variables with parameter $a>0$. As it is mentioned in \cite[p 506]{Gru}, by computation of the characteristic functions, we deduce that:
$$\frac{\sum_{k=1}^{n}{L_{[U_k,V_k]}}}{n}-\frac{2a}{\pi}\log n-\frac{2a}{\pi}$$
converges in law to $a\mathcal{A}$ where $\mathcal{A}$ is a totally asymmetric Cauchy random variable with characteristic function:
$$\phi(u)=\exp \left(-|u|\big(1+\frac{2i}{\pi}sgn(u)\log|u|\big)\right).$$
So, if $C>\frac{2a}{\pi}$, then:
\begin{align*}
\Prob_{x_0}&\left(\frac{\sum_{k=1}^{n}{L_{[U_k,V_k]}}}{n\log n}\leq C\right)=\Prob_{x_0}\left(\frac{\sum_{k=1}^{n}{L_{[U_k,V_k]}}}{n}\leq C\log n\right)\\
&=\Prob_{x_0}\left(\frac{\sum_{k=1}^{n}{L_{[U_k,V_k]}}}{n}-\frac{2a}{\pi}\log n-\frac{2a}{\pi}\leq (C-\frac{2a}{\pi})\log n -\frac{2a}{\pi}\right)\\
&=\Prob_{x_0}\left(a\mathcal{A}\leq (C-\frac{2a}{\pi})\log n -\frac{2a}{\pi}\right)\underset{n\to\infty}\longrightarrow 1
\end{align*}

By lemma \ref{lemmelongueur}, almost surely:
$$\frac{L_{[0,U_1]}+\sum_{k=1}^{n-1}{L_{[V_k,U_{k+1}]}}}{n\log n}\underset{n\to\infty}\longrightarrow 0.$$
This implies that, for every $\varepsilon>0$:
$$\Prob_{x_0}\left(\frac{L_{[0,U_1]}+\sum_{k=1}^{n-1}{L_{[V_k,U_{k+1}]}}}{n\log n}>\varepsilon\right)\underset{n\to\infty}\longrightarrow 0.$$

As $t\in[V_n,U_{n+1}]$ for some $n$, we have for every $\varepsilon>0$:
\begin{align*}
\Prob_{x_0}&\left(\frac{L_{[V_n,t]}}{n\log n}>\varepsilon\right)\leq\Prob_{x_0}\left(\frac{\sup_{s\in[V_n,U_{n+1}]}L_{[V_n,s]}}{n\log n}>\varepsilon\right)\\
&\leq\sup_{z\in\cup_{\alpha}\partial D_{\alpha}}\Prob_z\left(\frac{\sup_{s\in[0,U_{1}]}L_s}{n\log n}>\varepsilon\right)\underset{n\to\infty}\longrightarrow 0.
\end{align*}

Let $C>\frac{2a}{\pi}$ and $\varepsilon>0$ small enough so that $C-\varepsilon>\frac{2a}{\pi}$. We have:
\begin{align*}
\Prob_{x_0}&\left(\frac{L_t}{n\log n}\leq C\right)\\
&\geq\Prob_{x_0}\left(\frac{\sum_{k=1}^{n}{L_{[U_k,V_k]}}}{n\log n}+\frac{L_{[0,U_1]}+\sum_{k=1}^{n-1}{L_{[V_k,U_{k+1}]}}+L_{[V_n,t]}}{n\log n}\leq C\right)
\end{align*}
which is bigger than the probability of the intersection of the events:
$$\left\{\frac{\sum_{k=1}^{n}{L_{[U_k,V_k]}}}{n\log n}\leq C-\varepsilon\right\}$$
and
$$\left\{\frac{L_{[0,U_1]}+\sum_{k=1}^{n-1}{L_{[V_k,U_{k+1}]}}+L_{[V_n,t]}}{n\log n}\leq \varepsilon\right\}.$$
The probability of the two previous events going to $1$ when $n$ goes to $\infty$, we deduce that:
\begin{equation}
\label{equation10}
\Prob_{x_0}\left(\frac{L_t}{n\log n}\leq C\right)\underset{n\to\infty}\longrightarrow 1.
\end{equation}

Let $C'>\frac{2a}{\pi}\cdot\frac{1}{S}$ where $S$ is the limit of $\frac{V_n}{n}$. Write $C'=\frac{2a\alpha_0}{\pi}\cdot\frac{\alpha_1}{S}$ with $\alpha_0>1$ and $\alpha_1>1$. By assumption $t\in[V_n,U_{n+1}]$ for some $n$, hence:


\begin{align*}
\Prob_{x_0}\left(\frac{L_t}{t\log t}\leq C'\right)&\geq \Prob_{x_0}\left(\frac{L_t}{V_n\log V_n}\leq C'\right)\\
&=\Prob_{x_0}\left(\frac{L_t}{n\log n}\cdot \frac{n\log n}{V_n\log V_n}\leq C'\right)\\
&\geq \Prob_{x_0}\left(\left\{\frac{L_t}{n\log n}\leq\frac{2a\alpha_0}{\pi}\right\}\cap \left\{\frac{n\log n}{V_n\log V_n}\leq \frac{\alpha_1}{S}\right\}\right)\\
&\geq 1-\Prob_{x_0}\left(\frac{L_t}{n\log n}>\frac{2a\alpha_0}{\pi}\right)-\Prob_{x_0}\left(\frac{n\log n}{V_n\log V_n}> \frac{\alpha_1}{S}\right)
\underset{n\to\infty}\longrightarrow 1,
\end{align*}
by inequality \eqref{equation10} and the fact that, almost surely $\frac{n\log n}{V_n\log V_n}\underset{n\to\infty}\longrightarrow \frac{1}{S}$.

Now, if $t\in[U_n,V_n]$ for some $n$, then we have:
$$L_t\leq \sum_{k=1}^{n-1}{L_{[U_k,V_k]}}+L_{[0,U_1]}+\sum_{k=1}^{n-1}{L_{[V_k,U_{k+1}]}}+L_{[U_n,t]}$$
and we have for every $\varepsilon>0$:
\begin{align*}
\Prob_{x_0}&\left(\frac{L_{[U_n,t]}}{n\log n}>\varepsilon\right)\leq\Prob_{x_0}\left(\frac{\sup_{s\in[U_n,V_n]}L_{[U_n,s]}}{n\log n}>\varepsilon\right)\\
&=\Prob_z\left(\frac{\sup_{s\in[0,V]}L_s}{n\log n}>\varepsilon\right)\underset{n\to\infty}\longrightarrow 0.
\end{align*}
where $z$ is any point belonging to some $\partial d_{\alpha}$ and $V$ denote the hitting time of $\partial D_{\alpha}$ and we conclude by the same reasoning as in the case $t\in[V_n,U_{n+1}]$.

\section{Proof of theorem \ref{theoHolonomie}.}\label{section4}

Firstly, remark that there exists $\lambda>0$ such that for any continuous path $\omega:[0,\infty[\to \Sigma\setminus\mathcal{S}$ with $\omega(0)=a$, we have: 
\begin{equation}
\label{lambda}
||Hol_t(\omega)||\leq e^{\lambda L_t(\omega)}.
\end{equation}
Then from Item $1$ in Proposition \ref{mainpro} and Item $2$ in theorem \ref{Gruet}, we deduce immediatly Items $1$ and $2$ in theorem \ref{theoHolonomie}.  \newline

The proof of Item $3$ is more complicated essentially because, in general, we can find elements $X$ of $\pi_1(\Sigma\setminus\mathcal{S})$ which write as reduced words with arbitrarily big length and such that $||\rho(X)||$ is small. 

Write $\Sigma\setminus\mathcal{S}=\Hyp/\Gamma$ and let $X_{N_k}$ be the sequence of elements of the fundamental group $\Gamma$ defined by the discretization procedure. By hypothesis, there exists a loop $\alpha_0$ around a puncture $p_0$ such that $\gamma_0=\rho(\alpha_0)$ is a hyperbolic element. We denote by $y_0$ (resp. $z_0$) the repulsive (resp. attractive) fixed point of $\gamma_0$ and by $s(\rho(X_{N_k}))$ (resp. $n(\rho(X_{N_k}))$) the south (resp. north) pole of $\rho(X_{N_k})$ coming from the Cartan decomposition of $\rho(X_{N_k})$. We have:

\begin{lem}\label{lemm2}
Let $(\beta_k)_{k\in\mathbb{N}}$ be a sequence with $\beta_k\underset{k\to\infty}\longrightarrow 0$ and let $(E_k)_{k\geq 1}$ be the sequence of events defined by:
$$E_k=\{d(s(\rho(X_{N_{k-1}})),z_0)\geq \beta_k\}.$$
We have:
$$\tProb(E_k)\underset{k\to\infty}\longrightarrow 1.$$
\end{lem}

\begin{proof}
To simplify the notations, we denote by $Y_k=\rho(X_{N_{k-1}})$. We have:
\begin{align*}
\tProb\big(d(s(Y_k),z_0)<\beta_k\big)&=\tProb\big(d(n(Y_k^{-1}),z_0)<\beta_k\big)\\
&=\tProb\big(d(n(Y_k),z_0)<\beta_k\big)
\end{align*}
The first equality is due to the fact that $s(Y_k)=n(Y_k^{-1})$ and the second one is due to the fact that for every $Y\in\rho(\Gamma)$, $\tProb(Y_k=Y)=\tProb(Y_k^{-1}=Y)$ which is a consequence of the fact that the measure $\widetilde{\mu}$ is symmetric (see proposition \ref{propo}).

Note that there exists $Z(\omega)\in\CPun$ such that almost surely $n(Y_k)\underset{k\to\infty}\longrightarrow Z$. Moreover $Z$ is distributed on $\CPun$ with respect to a $\widetilde{\mu}$-stationary measure $\nu$.

Let $\varepsilon>0$. As $\nu$ is non atomic, there exists $\alpha>0$ such that, for $k$ big enough, $\nu(D(z_0,2\alpha))<\varepsilon$. Moreover, for $k$ big enough, we have $\beta_k<\alpha$ and hence the event $\{d(n(Y_k)),z_0)<\beta_k\}$ is contained in $\{d(n(Y_k)),z_0)<\alpha\}$ which is equal to the union of the events
$$\{d(n(Y_k),z_0)<\alpha\}\cap\{d(z_0,Z)<2\alpha\}$$
and
$$\{d(n(Y_k),z_0)<\alpha\}\cap\{d(z_0,Z)\geq 2\alpha\}.$$

This implies that:
$$\tProb\big(d(n(Y_k),z_0)<\beta_k\big)\leq \tProb\big(d(z_0,Z)<2\alpha\big)+\tProb\big(d(n(Y_k),Z)>\alpha\big).$$

Almost surely, $n(Y_k)\underset{k\to\infty}\longrightarrow Z$. So, for $k$ big enough, $\tProb\big(d(n(Y_k),Z)>\alpha\big)<\varepsilon$. And we also have $\tProb\big(d(z_0,Z)<2\alpha\big)=\nu(D(z_0,2\alpha))<\varepsilon$.

This concludes the proof of the lemma.

\end{proof}

Let $t\mapsto h(t)$ be as in the hypothesis of theorem \ref{theoHolonomie} Item $3$. Let $t\mapsto \widetilde{h}(t)$ be a nondecreasing positive function such that $\lim\limits_{t\to\infty}\widetilde{h}(t)=+\infty$, $\lim\limits_{t\to\infty}\frac{\widetilde{h}(t)}{h(t)}=+\infty$ and $\int_1^{\infty}\frac{dt}{\widetilde{h}(t)}=+\infty$.

\begin{lem}\label{lemm1}
Let $(B_k)_{k\geq 1}$ be the sequence of events defined by:
$$B_k=\{X_{N_{k-1}}^{-1}X_{N_{k}}=\alpha_0^n \text{ with } n\geq \widetilde{h}(k)\}$$
We have:
$$\tProb(B_k)\geq \frac{cst}{\widetilde{h}(k)}.$$
\end{lem}

\begin{proof}
The proof is a little bit technical but the lemma is essentially a consequence of the following simple fact. Consider an Euclidean Brownian motion in $\mathbb{C}$ starting at $z=\frac{1}{2}$. Let $V$ be the hitting time of $|z|=1$. For any Brownian path $\omega$ starting at $z=\frac{1}{2}$, append $\omega_{|[0,V(\omega)]}$ with the geodesic going back to $z=\frac{1}{2}$. With probability $1$, $\omega_{|[0,V(\omega)]}$ does not visit $z=0$ and $\omega(V(\omega))\neq-1$. Hence, this defines an element $\alpha(\omega)$ of $\pi_1(\mathbb{D}^*)$. Denoting by $\alpha_0$ a generator of $\pi_1(\mathbb{D}^*)$, there exists $n(\omega)\in\mathbb{Z}$ such that $\alpha(\omega)=\alpha_0^{n(\omega)}$. It is not difficult to see that $n(\omega)$ is distributed with respect to a Cauchy law. By a simple calculation, this implies that:
$$\Prob(n(\omega)\geq l)\sim_{l\to\infty}\frac{1}{l}.$$
 
Now, let us give a formal proof. Denote by $d_0$ and $D_0$ the two closed horodiscs around the puncture $p_0$ corresponding to the loop $\alpha_0$. 
Let $B_k'=\{X_{N_1}=\alpha_0^n \text{ with } n\geq \widetilde{h}(k)\}$. We have:
$$\tProb(B_k')=\tProb(B_k).$$
Moreover,
$$B_k'\supset\{X_{1}=\alpha_0^n \text{ with } n\geq \widetilde{h}(k)\}\cap\{N_1=1\}$$
By the definition of the sequence $N_k$ and inequality \eqref{kappa}, we have:
$$\{N_1=1\}=\{\alpha_1\leq\kappa_1\}\supset\{\alpha_1\leq\frac{1}{C^2}\},$$
we deduce that:
$$\tProb(B_k)\geq\frac{1}{C^2}\cdot\Prob_m\big(\{X_{1}=\alpha_0^n \text{ with } n\geq \widetilde{h}(k)\}\big),$$
where  $m$ is the Lebesgue measure on $\partial V_{Id}$.

Recall that, for the discretization procedure, we fixed a base point $x_0\in\Hyp$. As at the beginning of the proof of proposition \ref{mainpro}, let us fix a Dirichlet fundamental domain $D$ associated with this base point and denote by $p:\Hyp\to\Hyp/\Gamma=\Sigma\setminus\mathcal{S}$ the universal covering map. For every $X\in\pi_1(\Sigma\setminus\mathcal{S})$, we denoted by $F_{X}=X\cdot\overline{D_{\Hyp}(x_0,\delta)}$ and $V_{X}=X\cdot\overline{D_{\Hyp}(x_0,\delta')}$ with $\delta<\delta'$. We can choose $\delta'$ small enough so that $V_{Id}$ is included in the interior of the Dirichlet fundamental domain.
On the probability space $(\Omega,\Prob_m=\int\Prob_x\cdot dm(x))$, define the set $B_k''$ as the set of Brownian paths 
\begin{enumerate}
\item hitting $p^{-1}(d_0)$ before hitting the boundary of $D$,
\item then hitting the boundary of $p^{-1}(D_0)$ at a point of $\alpha_0^n\cdot D$ with $n\geq\tilde{h}(k)$,
\item and then hitting $F_{\alpha_0^n}$ before hitting $\cup_{\alpha\neq\alpha_0^n}F_{\alpha}$.
\end{enumerate}
By construction, we have $\Prob_m\big(\{X_{1}=\alpha_0^n \text{ with } n\geq \widetilde{h}(k)\}\big)\geq\Prob_m(B_k'')$. Using the notation $T_A$ for the hitting time of a closed set $A$ and recalling that $L_T$ stands for the length of the homotopic word between the times $0$ and $T$, we have by the strong Markov property:
\begin{align*}
\Prob_m(B_k'')\geq &\sum_{n\geq \widetilde{h}(k)}\Prob_m\big(T_{p^{-1}(d_0)}\leq T_{\partial D}\big)\cdot \frac{1}{2}\inf_{y\in D\cap \partial p^{-1}(d_0)}\Prob_y\big(L_{T_{\partial p^{-1}(D_0)}}=n\big)\\
&\cdot \inf_{z\in \alpha_0^n D\cap\partial p^{-1}(D_0)}\Prob_z\big(T_{F_{\alpha_0^n}}\leq T_{\cup_{\gamma\neq \alpha_0^n}F_{\gamma}}\big).
\end{align*}
As $\alpha_0^n$ is an isometry, we have:
$$\inf_{z\in \alpha_0^n D\cap\partial p^{-1}(D_0)}\Prob_z\big(T_{F_{\alpha_0^n}}\leq T_{\cup_{\gamma\neq \alpha_0^n}F_{\gamma}}\big)=\inf_{z\in D\cap\partial p^{-1}(D_0)}\Prob_z\big(T_{F_{Id}}\leq T_{\cup_{\gamma\neq Id}F_{\gamma}}\big),$$
and we deduce:

\begin{align*}
\Prob_m(B_k'') &\geq\Prob_m\big(T_{D\cap p^{-1}(d_0)}\leq T_{\partial D}\big)\cdot \inf_{z\in D\cap\partial p^{-1}(D_0)}\Prob_z\big(T_{F_{Id}}\leq T_{\cup_{\gamma\neq Id}F_{\gamma}}\big)\\
&\hspace{2cm}\cdot\frac{1}{2}\sum_{n\geq \widetilde{h}(k)}\inf_{y\in D\cap \partial p^{-1}(d_0)}\Prob_y\big(L_{T_{\partial p^{-1}(D_0)}}=n\big)\\
&=\Prob_m\big(T_{D\cap p^{-1}(d_0)}\leq T_{\partial D}\big)\cdot \inf_{z\in D\cap\partial p^{-1}(D_0)}\Prob_z\big(T_{F_{Id}}\leq T_{\cup_{\gamma\neq Id}F_{\gamma}}\big)\\
&\hspace{2cm}\cdot \frac{1}{2} \inf_{y\in D\cap \partial p^{-1}(d_0)}\Prob_y\big(L_{T_{\partial p^{-1}(D_0)}}\geq\widetilde{h}(k)\big).
\end{align*}

One convinces easily that there is a constant $C>0$ such that:
$$\Prob_m\big(T_{D\cap p^{-1}(d_0)}\leq T_{\partial D}\big)\cdot\inf_{z\in D\cap\partial p^{-1}(D_0)}\Prob_z\big(T_{F_{Id}}\leq T_{\cup_{\gamma\neq Id}F_{\gamma}}\big)\geq C$$
In order to evaluate the last term $\inf_{y\in D\cap\frac{1}{2} p^{-1}(d_0)}\Prob_y\big(L_{T_{\partial p^{-1}(D_0)}}\geq\widetilde{h}(k)\big)$, it is easier to think the Brownian motion running on the surface $\Sigma$ instead of thinking it on its universal cover $\Hyp$. As in the proof of proposition \ref{mainpro}, we identify $D_0$ (resp. $d_0$) with $D_{eucl}(0,1)\setminus\{0\}\subset \mathbb{C}$  (resp. $D_{eucl}(0,r)\setminus\{0\}$ with $0<r<1$) by some bijective conformal map. For an Euclidean Brownian motion starting at a point $\tau$ of $\partial D_{eucl}(0,r)$, let $\theta_t$ be a local continuous determination of the winding around $0$. If $V$ denotes the hitting time of $\partial D_{eucl}(0,1)$, then we have:
$$\inf_{y\in D\cap p^{-1}(d_0)}\Prob_y\big(L_{T_{\partial p^{-1}(D_0)}}\geq\widetilde{h}(k)\big)=\Prob_{\tau}\big(\theta_V\geq 2\pi\widetilde{h}(k)\big).$$ As the law of $\theta_V$ is that of a Cauchy law with parameter $a=\log\frac{1}{r}$ (i.e. the distribution law of $\theta_V$ is $\frac{1}{\pi}\cdot\frac{a}{a^2+x^2}$), a simple calculation gives that
$$\Prob_{\tau}\big(\theta_V\geq 2\pi\widetilde{h}(k)\big)\sim\frac{1}{\widetilde{h}(k)}$$
and the lemma is proved.

\end{proof}

\paragraph{End of the proof of the theorem.} 
We start with the following general result from probability theory:
\begin{lem}\label{lemmeproba}
Let $(A_n)_{n\geq 1}$ and $(B_n)_{n\geq 1}$ be two sequences of events. Assume that:
\begin{enumerate}
\item The two sequences $(A_n)_{n\geq 1}$ and $(B_n)_{n\geq 1}$ are independent.
\item Almost surely, $B_n$ is realized infinitely often.
\item $\Prob(A_n)\underset{n\to\infty}\longrightarrow 1$.
\end{enumerate}
Then, almost surely $A_n\cap B_n$ is realized infinitely often.
\end{lem}

\begin{proof}
Almost surely, $B_n$ is realized infinitely often. Hence, for almost every $\omega$, there exists a subsequence $N_k(\omega)$ such that $\omega\in B_{N_k(\omega)}$ for every $k\in\mathbb{N}^*$. Define the event $C_k=\{\omega\in A_{N_k(\omega)}\}$. We are going to prove that $\Prob(C_k)\underset{k\to\infty}\longrightarrow 1$. It is an easy exercice to check that this implie that almost surely, $C_k$ is realized infinitely often. The result follows immediatly.
\begin{align*}
\Prob(C_k)&=\sum_{l\geq 1}\Prob(\{\omega\in A_{N_k(\omega)}\}\cap \{N_k(\omega)=l\})\\
          &=\sum_{l\geq 1}\Prob(A_l\cap \{N_k(\omega)=l\}).
\end{align*}
The sequences $A_n$ and $B_n$ being independent, we deduce that the sequences $A_n$ and $N_k$ are also independent. And then, we get:
$$\Prob(C_k)=\sum_{l\geq 1}\Prob(A_l)\cdot\Prob (N_k=l).$$
Let $\varepsilon>0$. By the third hypothesis, there exist $l_0\in\mathbb{N}^*$ such that for every $l\geq l_0$, we have: $\Prob(A_l)\geq 1-\varepsilon$. Take $k\geq l_0$, we have:
$$\Prob(C_k)=\sum_{l= 1}^{l_0-1}\Prob(A_l)\cdot\Prob (N_k=l)+\sum_{l\geq l_0}\Prob(A_l)\cdot\Prob (N_k=l).$$
As $(N_k)$ is a subsequence, we have $N_k\geq k$. So if $l\leq l_0-1$ and $k\geq l_0$, we have: $\Prob (N_k=l)=0$. Hence we get for $k\geq l_0$:
$$\Prob(C_k)\geq (1-\varepsilon)\cdot \sum_{l\geq l_0}\Prob (N_k=l)=1-\varepsilon.$$
\end{proof}

By lemma \ref{lemm1}, we have $\tProb(B_k)\geq \frac{cst}{\widetilde{h}(k)}$. As $\sum_{k\geq1}\frac{cst}{\widetilde{h}(k)}=\infty$ and the $B_k$ are independent, we deduce by Borel-Cantelli lemma that almost surely $B_k$ is realized infinitely often. 

On the other hand, define $A_k$ by:
$$A_k=\{D\big(z_0,C_1|\lambda|^{-\widetilde{h}(k)}\big)\cap D\big(s(\rho(X_{N_{k-1}})),||\rho(X_{N_{k-1}})||^{-1}\big)=\varnothing.\}$$
where $\lambda$ is the eigenvalue of $\gamma_0$ with $|\lambda|>1$ and $C_1$ is the positive constant associated with $\gamma_0$ given by lemma \ref{lemmetechnique3}. As $||\rho(X_{N_k})||\underset{k\to\infty}\longrightarrow\infty$ and $\widetilde{h}(k)\underset{k\to\infty}\longrightarrow\infty$, we deduce that:
$$\tProb(A_k)= \tProb\left(d(z_0,s(\rho(X_{N_{k-1}})))>C_1|\lambda|^{-\widetilde{h}(k)}+||\rho(X_{N_{k-1}})||^{-1}\right)\underset{k\to\infty}\longrightarrow 1,$$
 by lemma \ref{lemm2}.

As $(A_k)_{k\geq 1}$ and $(B_k)_{k\geq 1}$ are independent, we deduce from lemma \ref{lemmeproba} that almost surely $A_k\cap B_k$ is realized infinitely often. So, almost surely, for infinitely many values of $k$, there exists an integer $n\geq \widetilde{h}(k)$ such that:
\begin{align}
\label{e1} &1.\;\,\rho(X_{N_{k}})=\rho(X_{N_{k-1}})\cdot\gamma_0^n\\
\label{e2} &2.\;\, D\big(z_0,C_1|\lambda|^{-\widetilde{h}(k)}\big)\cap D\big(s(\rho(X_{N_{k-1}})),||\rho(X_{N_{k-1}})||^{-1}\big)=\varnothing.
\end{align}
Fix $k$ satisfying the two previous items. By lemma \ref{lemmetechnique3}, we have:
\begin{equation}
\label{eqqq}
\gamma_0^n\left(D\big(y_0,C_1|\lambda|^{-\widetilde{h}(k)}\big)^c\right)\subset \overline{D\big(z_0,C_1|\lambda|^{-\widetilde{h}(k)}\big)}.
\end{equation}
Let $C_2\geq 1$ be a constant satisfying lemma \ref{lemmetechnique4}. We have:
\begin{align*}
\rho(X_{N_{k}})\left(D\big(y_0,C_1|\lambda|^{-\widetilde{h}(k)}\big)^c\right)&=\rho(X_{N_{k-1}})\cdot\gamma_0^n\left(D\big(y_0,C_1|\lambda|^{-\widetilde{h}(k)}\big)^c\right)\\
&\subset\rho(X_{N_{k-1}})\left(\overline{D\big(z_0,C_1|\lambda|^{-\widetilde{h}(k)}\big)}\right)\\
&\subset \overline{D\big(\rho(X_{N_{k-1}}) z_0\;,\;C_1 C_2|\lambda|^{-\widetilde{h}(k)}\big)}.
\end{align*}
The first line comes from \eqref{e1}, the second one from \eqref{eqqq} and the third one from \eqref{e2} and lemma \ref{lemmetechnique4}.
As $C_2\geq 1$, this implies that:
$$\rho(X_{N_{k}})\left(D\big(y_0,C_1 C_2|\lambda|^{-\widetilde{h}(k)}\big)^c\right)\subset \overline{D\big(\rho(X_{N_{k-1}}) z_0\;,\;C_1 C_2|\lambda|^{-\widetilde{h}(k)}\big)},$$
which implies, by lemma \ref{lemmetechnique2}:
$$||\rho(X_{N_{k}})||\geq \frac{\sqrt{1-(3 C_1 C_2|\lambda|^{-\widetilde{h}(k)})^2}}{3C_1 C_2|\lambda|^{-\widetilde{h}(k)}}\geq A e^{B\widetilde{h}(k)}$$
for some positive constants $A$ and $B$.
This implies that:
$$\frac{1}{h(k)}\cdot\log||\rho(X_{N_{k}})||\geq \frac{\log A}{h(k)}+B\cdot\frac{\widetilde{h}(k)}{h(k)}.$$
The last inequalitie being true for infinitely many values of $k$ ans as $\frac{\widetilde{h}(k)}{h(k)}\underset{k\to\infty}\longrightarrow\infty$, we deduce:
$$\limsup_{k\to\infty}\frac{1}{h(k)}\cdot\log||\rho(X_{N_{k}})||=\infty.$$
The theorem follows from the facts that $Hol_{S_{N_k}}=(\rho(X_{N_{k}}))^{-1}$, $||\rho(X_{N_{k}})||=||\rho(X_{N_{k}})^{-1}||$ and $\frac{S_{N_k}}{k}\underset{k\to\infty}\longrightarrow T>0$.

\section{Proof of theorem \ref{theoprincipal}.}\label{section5}
This theorem was proved in \cite[Theorem 6.1]{Hus} in the non-hyperbolic case. From now on, assume that we are in the hyperbolic case (i.e. there exists a loop around a puncture with hyperbolic monodromy). The idea of the proof in this case is the same as in the non-hyperbolic case. Hence, for more details, we refer the reader to \cite{Hus}.

Let $\widetilde{\omega}=(\omega,\alpha) \in \widetilde{\Omega}$. The path $\omega$ can be written as an infinite concatenation of paths:
$$\omega=\beta_0\ast\omega_0\ast\beta_1\ast\omega_1\ast\cdot\cdot\cdot,$$
where $\beta_0=\omega_{|[0,S_{N_0}]}$, for $k\geq 0$, $\omega_k=\omega_{|[S_{N_k},R_{N_{k+1}}]}$ and for $k\geq 1$, $\beta_k=\omega_{|[R_{N_k},S_{N_k}]}$. Let $c_k(t)=X_{N_k}^{-1}\cdot\omega_k(t-S_{N_k})$. The $(c_k)_{k\in\mathbb{N}}$ form a family of portions of Brownian paths independent and identically distributed: the distibution law of their starting point is the hitting measure of $\partial V_{Id}=\partial D(x_0,\delta')$ for a Brownian motion starting at $x_0$ and they are stopped at time $R_{N_{k+1}}-S_{N_k}$.
\[\omega=\beta_0\ast X_{N_0}c_0\ast\beta_1\ast X_{N_1}c_1\ast\cdot\cdot\cdot\]
Because of the $\rho$-equivariance, we have:
\[\mathcal{D}(\omega)=\mathcal{D}(\beta_0)\ast \rho(X_{N_0})\mathcal{D}(c_0)\ast\mathcal{D}(\beta_1)\ast \rho(X_{N_1})\mathcal{D}(c_1)\ast\cdots\]

Now, we are going to push forward the right random walk $X_{N_k}$ by $\rho$ in order to obtain a right random walk in the monodromy group $\rho(\Gamma)$. For this, we write $\widetilde{\mu}=\rho_{\ast}\mu$ (where $\mu$ is the probability measure in $\Gamma$ defined by the discretization procedure) and $Y_{N_k}=\rho(X_{N_k})$. The process $(Y_{N_k})_{k\geq 0}$ is a realisation of a right random walk in $\rho(\Gamma)$ with law $\widetilde{\mu}$.\\

We have:

\begin{lem}\label{lemmeBk}
Let $\lambda'>\lambda\cdot C\cdot T$ where the constants $\lambda$, $C$ and $T$ are given respectively by \eqref{lambda}, propositions \ref{mainpro} and \ref{propo} and let:
$$A_k=\left\{||Y_{N_k}||\leq e^{\lambda'k\log k}\right\}.$$
Then, $\tProb(A_k)\underset{k\to\infty}\longrightarrow 1.$
\end{lem}

\begin{proof}
By \eqref{lambda}, we have $||Y_{N_k}||=||Hol_{S_{N_k}}||\leq e^{\lambda L_{S_{N_k}}}$. Hence:
\begin{align*}
\tProb(A_k)&\geq\tProb\left(e^{\lambda L_{S_{N_k}}}\leq e^{\lambda'k\log k}\right)\\
&=\tProb\left(\frac{L_{S_{N_k}}}{k\log k}\leq \frac{\lambda'}{\lambda}\right)\\
&=\tProb\left(\frac{L_{S_{N_k}}}{S_{N_k}\log S_{N_k}}\cdot\frac{S_{N_k}\log S_{N_k}}{k\log k}\leq C\cdot\frac{\lambda'}{\lambda C T}T\right)\\
&\geq\tProb\left(\left\{\frac{L_{S_{N_k}}}{S_{N_k}\log S_{N_k}}\leq C\right\}\cap\left\{\frac{S_{N_k}\log S_{N_k}}{k\log k}\leq \frac{\lambda'}{\lambda C T}T\right\}\right)\\
&\geq 1-\tProb\left(\frac{L_{S_{N_k}}}{S_{N_k}\log S_{N_k}}> C\right)-\tProb\left(\frac{S_{N_k}\log S_{N_k}}{k\log k}> \frac{\lambda'}{\lambda C T}T\right)
\end{align*}
As $\frac{S_{N_k}}{k}\underset{k\to\infty}\longrightarrow T$ and $\frac{\lambda'}{\lambda C T}>1$, we deduce that $\tProb\left(\frac{S_{N_k}\log S_{N_k}}{k\log k}> \frac{\lambda'}{\lambda C T}T\right)\underset{k\to\infty}\longrightarrow 0$. And, by proposition \ref{mainpro}, we have $\tProb\left(\frac{L_{S_{N_k}}}{S_{N_k}\log S_{N_k}}> C\right)\underset{k\to\infty}\longrightarrow 0$. The proof is complete.
\end{proof}

\begin{lem}\label{lemmeEk}
Let:
$$B_k=\{\mathcal{D}(c_{k})\cap D(s_{k},e^{-2\lambda'k\cdot\log k})\neq \varnothing\}\cap\{\mathcal{D}(c_{k})\cap (D(s_{k},||Y_{N_k}||^{-1}))^c \neq \varnothing.\}$$
where $s_k$ is the south pole of $Y_{N_k}$ coming from Cartan's decomposition. Then almost surely, $B_k$ occurs infinitely often.
\end{lem}

\begin{proof}
Let $\mathcal{F}_k$ be the canonical filtration associated with the sequence $(B_k)_{k\geq 0}$ (i.e $\mathcal{F}_k$ is the $\sigma$-field generated by $B_0,\cdots,B_k$). By the second Borel-Cantelli lemma (see \cite[Theorem 5.3.2]{Dur}), it is enough to prove that almost surely:
$$\sum_{k\geq 0} \tProb(B_k/\mathcal{F}_{k-1})=\infty.$$
As $||Y_{N_k}||^{-1}\underset{k\to\infty}\longrightarrow 0$, the event $\{\mathcal{D}(c_{k})\cap (D(s_{k},||Y_{N_k}||^{-1}))^c \neq \varnothing.\}$ is certain for $k$ big enough. Hence:
$$\tProb(B_k/\mathcal{F}_{k-1})=\tProb(\mathcal{D}(c_{k})\cap D(s_{k},e^{-2\lambda'k\cdot\log k})\neq \varnothing/\mathcal{F}_{k-1}).$$
Note also that by construction, $s_k$ depends only on the set $X_{N_1},\cdots,X_{N_k}$ (i.e. it depends on the set $\gamma_{N_1},\cdots,\gamma_{N_{k}}$) and $c_k$ depends only on $X_{N_k}^{-1}X_{N_{k+1}}=\gamma_{N_{k+1}}$. As the $\gamma_{N_i}$ are mutually independent, we deduce that $s_k$ and $c_k$ are independent. Hence almost surely:
$$\tProb\left(B_k/\mathcal{F}_{k-1}\right)\geq \underset{y\in\CPun}\inf\tProb\left(\mathcal{D}(c_{k})\cap D(y,e^{-2\lambda'k\log k})\neq \varnothing/\mathcal{F}_{k-1}\right).$$
As $\mathcal{D}(c_{k})$ and $\mathcal{F}_{k-1}$ are independent, we deduce that almost surely:
$$\tProb\left(\mathcal{D}(c_{k})\cap D(y,e^{-2\lambda'k\log k})\neq \varnothing/\mathcal{F}_{k-1}\right)=\tProb\left(\mathcal{D}(c_{k})\cap D(y,e^{-2\lambda'k\log k})\neq \varnothing\right).$$
As $\DD$ is onto, there is a compact $K\in\Hyp$ and a positive constant $\alpha$ such that for every $y\in\CPun$, and every $k$ big enough, $\DD^{-1}\left(D(y,e^{-2\lambda'k\cdot\log k})\right)\cap K$ contains a disc $D_{k,y}$ with radius $e^{-\alpha k\log k}$. Hence, 
\begin{align*}
\tProb(E_k/\mathcal{F}_{k-1})&\geq \inf_{y\in\CPun}\tProb\left(c_{k}\cap \DD^{-1}(D(y,e^{-2\lambda'k\log k})\neq \varnothing\right)\\
&\geq \inf_{y\in\CPun}\tProb\left( c_{k}\cap D_{k,y}\neq \varnothing\right)\\
&\geq \inf_{y\in\CPun}\Prob_{\varepsilon_{x_0}}\left(T_{D_{k,y}}\leq T_{\cup F_{\gamma}}\right)
\end{align*}
where $\varepsilon_{x_0}$ is the hitting measure of $\partial V_{Id}$ for a Brownian motion starting at $x_0$ and $T_A$ is the hitting time of $A$.
As for every $y\in\CPun$, the disc $D_{k,y}$ has radius $e^{-\alpha k\log k}$ and is included in the compact $K$, this last probability is bigger than $\frac{cst}{k\log k}$ (this is a direct adaptation of \cite[Lemma 6.6]{Hus}). As $\sum_{k\in\mathbb{N}^*}\frac{1}{k\log k}=\infty$, the proof is complete.

\end{proof}

\begin{pro}
Almost surely, for infinitely many values of $k$, we have:
\begin{enumerate}
\item $Y_{N_{k}}\big(D(s_{k},||Y_{N_{k}}||^{-1})^c\big)\subset \overline{D(n_{k},||Y_{N_{k}}||^{-1})},$
\item $\mathcal{D}(c_{k})\cap (D(s_{k},||Y_{N_k}||^{-1}))^c \neq \varnothing,$
\item $Y_{N_{k}}\big(\overline{D(s_{k},e^{-2\lambda' k\log k})}\big)\subset D\left(n_k, \frac{1}{2}\right)^c,$
\item $\mathcal{D}(c_{k})\cap D(s_{k},e^{-2\lambda'k\log k})\neq \varnothing$.
\end{enumerate}
\end{pro}

\begin{proof}
By lemmas \ref{lemmeBk}, \ref{lemmeEk} and \ref{lemmeproba}, we have that almost surely $A_k\cap B_k$ occurs infinitely often. Choose $k\in\mathbb{N}$ such that $A_k\cap B_k$ occurs. Then Item 2 and 4 are satisfied. We also have $||Y_{N_k}||\leq e^{\lambda'k\log k}$. Hence, applying \eqref{equasup2} in lemma \ref{lemmetechnique1}, we obtain:
\begin{align*}
Y_{N_k}\left(\overline{D(s_{k},e^{-2\lambda' k\cdot\log k})}\right)&\subset Y_{N_k}\left(\overline{D(s_{k},||Y_{N_{k}}||^{-2})}\right)\\
&=D\left(n_k, \frac{1}{2}\right)^c.
\end{align*}
So Item 3 is satisfied. To check that Item 1 is also true, let $\alpha_k=\frac{1}{\sqrt{1+||Y_{N_k}||^2}}$ (which is equivalent to $||Y_{N_k}||=\frac{\sqrt{1-\alpha_k^2}}{\alpha_k}$). Noticing that $\alpha_k\leq ||Y_{N_k}||^{-1}$ and using \eqref{equainf} in lemma \ref{lemmetechnique1}, we obtain:
\begin{align*}
Y_{N_{k}}\big(D(s_{k},||Y_{N_{k}}||^{-1})^c\big)&\subset Y_{N_{k}}\big(D(s_{k},\alpha_k)^c\big)\\
&=\overline{D(n_{k},\alpha_k)}\\
&\subset \overline{D(n_{k},||Y_{N_{k}}||^{-1})}.
\end{align*}
\end{proof}

\paragraph{End of the proof of Theorem \ref{theoprincipal}. }
The previous proposition implies that almost surely, for infinitely many values of $k$, the portion $\rho(X_{N_{k}})\mathcal{D}(c_{k})$ of the path $\mathcal{D}(\omega)$ visits $\overline{D(n_{k},||Y_{N_{k}}||^{-1})}$ and $D\left(n_{k},\frac{1}{2}\right)^c$. As $||Y_{N_{k}}||^{-1}\underset{k\to\infty}\longrightarrow 0$, this proves that $\mathcal{D}(\omega(t))$ does not have limit when $t$ goes to infinity.\\

\section{Proof of theorem \ref{theoRiccati}.}\label{section6}
We repeat the proof of \cite{Hus}.
Let $(M,\Pi,\FF,\Sigma,\rho,\mathcal{S})$ be a Riccati foliation. Let $F$ be a non invariant fiber, $s_0$ a section of $\Pi$ and denote by $S_0:=s_0(\Sigma)$. By sliding along the leaves, We can transport the unique complex projective structure defined on $F\cong\CPun$. We obtain a branched complex projective structure on $S_0\setminus s_0(\mathcal{S})\cong \Sigma\setminus\mathcal{S}$ whose monodromy representation is the monodromy representation of the foliation (the branched points are the points of $S_0$ where the foliation is tangent to $S_0$). By definition, if $p_0\in S_0$ is not a branched point and if $h:(F,p)\to (S_0,p_0)$ is a holonomy germ of the foliation, the analytic continuation of $h^{-1}$ defines a developing map $\DD_0$ of the complex projective structure on $S_0\setminus s_0(\mathcal{S})$. By assumption, $\DD_0$ is onto, then we can apply theorem \ref{theoprincipal}: for almost every Brownian path $\omega\in\Omega_{p_0}$, the path $(\DD_0(\omega(t)))_{t\geq 0}$ does not have limit when $t$ goes to $\infty$. By conformal invariance of Brownian motion, this is equivalent to the following: $h$ can be analytically continued along almost every Brownian path $\omega\in\Omega_{p}$. This finishes the proof of Item $1$.\\

For the proof of the second item, consider another section $s_1$ of $\Pi$ and denote by $S_1:=s_1(\Sigma)$ and let $h:(S_1,p_1)\to(S_0,p_0)$ be a holonomy germ.
\begin{description}
\item \textbf{First case:} $p_1$ (and then $p_0$) does not belong to $\Pi^{-1}(\mathcal{S})$. Then $h$ can be written $h=\mathcal{D}_0^{-1}\circ\mathcal{D}_1$ where $\mathcal{D}_1$ is also a developing map associated with the branched projective structure on $S_1\setminus s_1(\mathcal{S})$. By conformal invariance, the image of a generic Brownian path starting at $p_1$ by $\DD_1$ is a Brownian path in $\mathbb{C}\mathbb{P}^1$ (possibly stopped at a stopping time $T<\infty$ because we did not assume that $\DD_1$ is onto) along which $\DD_0^{-1}$ can be analytically continued by Item $1$.
\item \textbf{Second case:} $p_1$ (and then $p_0$) belongs to $\Pi^{-1}(\mathcal{S})$. Then, $h$ is the germ of a holonomy biholomorphism between two small neighborhood of $p_1$ in $S_1$ and $p_0$ in $S_0$. We conclude by the first case and the strong Markov property.
\end{description}

Nicolas Hussenot Desenonges\newline
  Instituto de Matem\'{a}tica, Universidade Federal do Rio de Janeiro, \newline
  Ilha do Fundao, 68530, CEP 21941-970, Rio de Janeiro, RJ, Brasil \newline
  \textit {e-mail}: nicolashussenot@hotmail.fr

\end{document}